\documentclass[12pt]{article}
\usepackage{amsmath}
\usepackage{amsfonts}
\usepackage{amssymb}
\usepackage{amsthm}
\usepackage{stackrel}

\usepackage{color}

\usepackage[T2A]{fontenc}			
\usepackage[T2A]{fontenc}			
\usepackage[utf8]{inputenc}

\usepackage{figsize}

\usepackage{graphics}

\newtheorem{thm}{Theorem}[section]
\newtheorem{co}{Collorary}[section]
\newtheorem{re}{Remark}[section]

\newcommand{\go}[1]{\mathfrak{#1}}

\newcommand{\R}{{\rm I}\kern-0.18em{\rm R}}
\newcommand{\1}{{\rm 1}\kern-0.25em{\rm I}}
\newcommand{\E}{{\rm I}\kern-0.18em{\rm E}}
\newcommand{\p}{{\rm I}\kern-0.18em{\rm P}}
\makeatletter

\def\@fnsymbol#1{\ensuremath{\ifcase#1\or a\or b\or c\or d\or \e\or f\or *\dagger 	\or \ddagger\ddagger 
		else\@ctrerr\fi}}

\title{On an arithmetical property of moments and cumulants}
\author{Ashot V. Kakosyan\thanks{Faculty of Economics and Management, Department of Mathematical  Modeling in EconomicsбYerevan State University, Yerevan, Armenia; e-mail: ashotkakosyan@ysu.am} and Lev B. Klebanov\thanks{Department of Probability and Mathematical Statistics, Charles University, Prague, 186 75, Czech Republic; e-mail (address for communications): klebanov@karlin.mff.cuni.cz }}
\date{}

\begin{document}
\maketitle
\begin{abstract}
The main result of the paper is the following.
	Let a non-degenerate distribution have finite moments $\mu_k$ of all orders $k=0,1,2,\ldots$. Then the sequence $\{\mu_k/k!, \; k=0,1,2,\ldots\}$ either contains infinitely many different terms or at most three. In the latter case, this sequence has the form $\{1,a,1-b,a,1-b,a,1-b, \ldots\}$ and corresponds to a distribution with the characteristic function \begin{equation*}\label{ eq0}
		f(t)=\frac{1+iat+bt^2}{1+t^2}, \quad \text{where} \;\; b\geq 0,\; \frac{1-a-b}{2}\geq 0,\; \frac{1+a-b}{2}\geq 0.
	\end{equation*}.


\end{abstract}

\section{Introduction}\label{sec1}
\setcounter{equation}{0}

The problems studied in this work are related to two observations. The first consists of considering the moments of the exponential distribution and Laplace distributions' moments. The second relates to the well-known Marcinkiewicz theorem (see, for example, \cite{Lin}). Let us explain this in more detail.

\begin{enumerate}
\item[A.] Consider the exponential distribution on the positive semi-axis with a 
unit scale parameter. Its moment $\mu_r$ of order $r$ is equal to $r!$, i.e. 
ratio $\mu_r/r! =1$ does not depend on the order of the moment. In other words, the set 
of different values of the fractions $\{\mu_r/r!,\;r=0,1,2,\ldots \}$ is 
finite (consists of one element). A similar situation occurs in the case of 
the Laplace distribution with zero mean and variance equal to 2. In this case,
$\mu_r/r! =(1+(-1)^r)/2$ and the set of different values of fractions 
$\{\mu_r/r!,\;r=0,1,2,\ldots \}$ is again finite (consists of two elements ). In 
this work, we will describe all probability distributions for which the set $\go M$ 
of different values of the fractions $\{\mu_r/r!,\;r=0,1,2,\ldots \}$ is finite.

\item[B.] Marcinkiewicz's well-known theorem states that if the cumulants of 
the $\{\varkappa_k\}$ distribution are equal to zero starting from a certain place, 
then this distribution is Gaussian. In this case, the set $\go K$ of different values 
of fractions $\{\varkappa_k/k!,\; k=0,1,2,\ldots\}$ of course. However, it is not 
yet clear whether an analog of Marcinkiewicz's theorem will be valid if the condition 
for the finiteness of the set of nonzero cumulants is replaced by the requirement for 
the finiteness of the set $\go K$. We will also provide an answer to this question below.
\end{enumerate}
If $f(t)$ is an analytic characteristic function, then its power series expansion has the form
\begin{equation}\label{eq*}
f(t) = \sum_{k=0}^{\infty}i^k \frac{\mu_k}{k!} t^k.
\end{equation}
From this, we see that if the set $\go M$ is finite, then the set of different coefficients of the series (\ref{eq*}) is finite and, conversely, if the set of different coefficients of the series (\ref{eq*}) is finite, then set $\go M$.
A similar connection exists between the set $\go K$ and the set of different coefficients 
of the power series expansion of $\log f(t)$.
The main result of the work for the case of moments is as follows:

\begin{thm}\label{th0}
Let a non-degenerate distribution have finite moments $\mu_k$ of all orders $k=0,1,2,\ldots$. Then the sequence $\{\mu_k/k!, \; k=0,1,2,\ldots\}$ either contains infinitely many different terms or at most three. In the latter case, this sequence has the form $\{1,a,1-b,a,1-b,a,1-b, \ldots\}$ and corresponds to a distribution with the characteristic function \begin{equation*}\label{ eq0}
f(t)=\frac{1+iat+bt^2}{1+t^2}, \quad \text{where} \;\; b\geq 0,\; \frac{1-a-b}{2}\geq 0,\; \frac{1+a-b}{2}\geq 0.
\end{equation*}.
\end{thm}

There is no such result for the case of cumulants.

\section{Property of moments}\label{sec2}
\setcounter{equation}{0}
Let us first consider nondegenerate random variables for which the set $\go M$ is finite.

It will be shown below that the characteristic functions of such random variables have the form of either
\begin{equation}\label{eq1}
f(t)=\frac{1\pm iat}{1\pm it},\quad \text{where} \; 0\leq a \leq 1,
\end{equation}
or
\begin{equation}\label{eq2}
f(t)=\frac{1+iat+bt^2}{1+t^2}, \quad \text{where} \;\; b\geq 0,\; \frac{1-a-b}{2}\geq 0,\; \frac{1+a-b}{2}\geq 0.
\end{equation}
Note that the function (\ref{eq1}) is a special case of (\ref{eq2}), but it is somewhat more convenient for us to consider these 
cases separately.

In addition to the question of moments, we study a similar question in the case of replacing the characteristic
functions by its logarithm.

\begin{thm}\label{th1}
Let $f(t)$ be the characteristic function of some non-degenerate random variable. Suppose that $f(t)$ is analytic in the disk \\
$|t|<r$, $r>0$ and
\begin{equation}\label{eq3}
f(t)=\sum_{k=0}^{\infty}a_k t^k.
\end{equation}
Let us also assume that the set $\go M$ is finite. Then $f(t)$ has the form (\ref{eq1}) or (\ref{eq2}), where $a,b$ are some constants. Conversely, if $f(t)$ has the form (\ref{eq1}) or (\ref{eq2}), then the set $\go M$ is finite.
\end{thm}
\begin{proof} As noted above, the finiteness of the set $\go M$ is equivalent to the finiteness of the set of different coefficients of the series (\ref{eq3}).
G. Szeg\"{o}'s theorem (\cite{Sz}, see also \cite{Bie}) states that an analytic function of the form (\ref{eq3}) is either non-extendable beyond the unit circle or can be written in the form
\begin{equation}\label{eq4}
f(t)=\frac{P(t)}{1-t^m},
\end{equation}
where $P(t)$ is a polynomial and $m$ is a non-negative integer. Note that the fraction on the right side (\ref{eq4}) is generally reducible.

Since $f(t)$ is an analytic characteristic function, it can be extended from the circle $|t|<r$ to the strip $|\tt{Im} \;t| < r$ (see \cite{Lin}) and the first alternative in G. Szeg\"{o}'s theorem is impossible. Consequently, $f(t)$ must have the form (\ref{eq4}), and, therefore, its possible singularities must be at the $m$th roots of unity. Thus, we can assume that $r=1$ and $f(t)$ is analytic in the strip $|\tt{Im} \;t| <$1. The latter is only possible if:
\begin{enumerate}
\item $f(t)$ has no singularities at all;
\item singularities $f(t)$ lie at one of the points $\pm i$;
\item singularities of $f(t)$ lie at both points $\pm i$.
\end{enumerate}
So, or
\begin{equation}\label{eqA}
f(t)=P(t),
\end{equation}
or
\begin{equation}\label{eq5}
f(t)=\frac{P(t)}{1\pm it}
\end{equation}
or finally
\begin{equation}\label{eq6}
f(t)=\frac{P(t)}{1+t^2},
\end{equation}
where $P(t)$ is some polynomial.
However, for any characteristic function $f(t)$, $|f(t)|\leq 1$ holds for all real $t$. Therefore, the degree of the polynomial $P(t)$ is equal to zero in the case of (\ref{eqA}), not higher than 1 in the case of (\ref{eq5}) and not higher than 2 in the case of (\ref{eq6}).
However, if the degree of the polynomial $P(t)$ in (\ref{eqA}) is equal to zero, then the corresponding distribution is degenerate, contrary to the assumption. Therefore, the case (\ref{eqA}) is impossible.

{\bf I) Let us first consider the case (\ref{eq5})}.
Since in this case the degree of the polynomial $P(t)$ is not higher than 1, then
\[ f(t) = \frac{A\pm Bt}{1 \pm i t}.\]
To clarify the constants, we will use the simplest properties of characteristic functions. If $f(t)$ is an arbitrary characteristic function, then $f(0)=1$. Therefore $A=1$. Next, $f(-t)=\overline{f(t)}$, where the overbar means the transition to a complex conjugate number. For $t\in \R^1$ we have
\[ f(-t) = \frac{1\mp Bt}{1\mp it}, \]
\[ \overline{f(t)} = \frac{1\pm \bar{B}t}{1 \mp it}.\]
Thus, for all $t \in \R^1$ the equality must be true
\[\frac{1\mp Bt}{1\mp it} = \frac{1\pm \bar{B}t}{1 \mp it}.\]
It follows that $B$ must be a purely imaginary number, i.e.
$B=\pm ia$, $a \in \R^1$. However,
\[ \frac{1 \pm iat}{1 \pm it} = a+\frac{1-a}{1 \pm it}.\]
Now we see that on the right side of the last equality there is a Fourier transform of a mixture of a degenerate distribution at zero with weight $a$ and an exponential distribution on the positive (+ sign) or negative (- sign) half-axis with weight $1-a$. This mixture must be a probability distribution. Therefore $0 \leq a \leq 1$.

{\bf II) Let us now consider the case (\ref{eq6})}. Since in this case the degree of the polynomial $P(t)$ is not higher than 2, and $f(0)=1$ then
\[f(t)= \frac{1+iat+bt^2}{1+t^2}=b+\frac{1-b+iat}{1+t^2}. \]
The condition $f(-t)=\overline{f(t)}$ shows that $a,b \in \R^1$.
It is easy to see that $f(t)$ can be represented as a mixture of Fourier transforms: 1) a degenerate distribution at zero with weight $b$; 2) exponential distribution on the negative semi-axis with weight $\frac{1-a-b}{2}$; 3) exponential distribution on the positive semi-axis with weight $\frac{1+a-b}{2}$. The sum of these weights is equal to 1. Due to the probabilistic meaning, these weights must be non-negative.
The first part of the Theorem has been proven.

However, it is easy to verify that the various coefficients of the power expansions of the functions (\ref{eq1}) and (\ref{eq2}) form a finite set, which completes the proof of the Theorem.
\end{proof}

As an addition to the proven Theorem, we note that for the function (\ref{eq2})
\[f(t)=\frac{1+iat+bt^2}{1+t^2}\]
$\go{M}=\{1,a,1-b\}$ is true. In more detail, $\mu_0=1$, $\mu_{2k-1}/(2k-1)! =a$ and $\mu_{2k}/2k!=1-b$, $k=1,2,\ldots$. Since the function (\ref{eq1}) is a special case of (\ref{eq2}), the set $\go{M}$ cannot consist of more than three elements.

So, as a Corollary we get the main result.

\begin{co}\label{co1} coincides with {\bf Theorem \ref{th0}}.
Let a non-degenerate distribution have finite moments $\mu_k$ of all orders $k=0,1,2,\ldots$. Then the sequence $\{\mu_k/k!, \; k=0,1,2,\ldots\}$ either contains infinitely many different terms or at most three. In the latter case, this sequence has the form $\{1,a,1-b,a,1-b,a,1-b, \ldots\}$ and corresponds to a distribution with a characteristic function (\ref{eq2}).
\end{co}
\begin{proof}
If the sequence $\{\mu_k/k!, \; k=0,1,2,\ldots\}$ contains only a finite set of distinct elements, then the corresponding characteristic function is analytic in some circle centered at zero. And, therefore, the required follows from Theorem \ref{th1} and its addition.
\end{proof}
\section{Property of cumulants}\label{sec3}
\setcounter{equation}{0}

In this section we will consider instead of the function $f(t)$ its logarithm
$g(t)=\log f(t)=\sum_{k=0}^{\infty}i^k \varkappa_k t^k/k!$. We assume a set $\go K$ of different fraction values $\{\varkappa_k/k!,\; k=0,1,2,\ldots\}$ of course. The following result is true.

\begin{thm}\label{th2}
Let a non-degenerate random variable have a characteristic function $f(t)$, analytic in the disk $|t|<r$, $r>0$. Let us assume that the set $\go K$ is finite for it. Then either
\begin{equation}\label{eqc1}
g(t)=-\sigma^2 t^2+iat, \quad \sigma \geq 0,\; a\in\R^1,
\end{equation}
or
\begin{equation}\label{eqc2}
g(t)=\frac{P(t)}{1\pm it},
\end{equation}
where $P(t)$ is a polynomial, or, finally,
\begin{equation}\label{eqc3}
g(t)=\frac{P(t)}{1+t^2},
\end{equation}
where, again, $P(t)$ is some polynomial.
\end{thm}
Note that (\ref{eqc2}) is a special case of (\ref{eqc3}).
\begin{proof}
The finiteness of the set $\go K$ means that the set of different coefficients of the power series expansion of the function $g(t)$ is finite. Szego's theorem shows that
\[ g(t) =\frac{P(t)}{1-t^m} \]
in view of the analytical prolongation of $f(t)$ and, therefore, $g(t)$ from the unit disk. As in the proof of Theorem \ref{th1}, we see that only the cases $g(t)=P(t)$, (\ref{eqc2}) and (\ref{eqc3}) are possible. However, if $g(t)=P(t)$, then the well-known Marcinkiewicz theorem shows that the degree of the polynomial is not higher than 2, and the characteristic function is normal, i.e. true (\ref{eqc1}). The theorem has been proven.
\end{proof}
Since the function $g(t)$ must not be bounded, we cannot find a bound on the degree of the polynomial $P(t)$. However, we immediately note that $1/(1+t^2)$ is a characteristic function, and therefore $1-1/(1+t^2)$ is the logarithm of some characteristic function. Multiplying it by the characteristic function of the standard normal law we see that
\[1-1/(1+t^2)-t^2/2= \frac{1+t^2-t^4}{2(1+t^2)} \]
again the logarithm of some characteristic function for which the set $\go K$ is finite. Thus, the fourth degree of the polynomial $P(t)$ in Theorem \ref{th2} is possible.

!!! Добавить о безграничной делимости и о ее отсутствии!!!

\begin{re}\label{re1}
Note that in (\ref{eqc3}) the degree of the polynomial $P(t)$ can be greater than four.
\end{re}
\begin{proof}
To prove it, it is enough to construct an appropriate example. We will do this in two stages. First, we determine the distribution on the non-negative semi-axis using its Laplace transform, and then move on to the required characteristic function of another random variable.

Let us define the function $h(s)$ by the equality
\begin{equation}\label{eqc4}
h(s) = \exp\{ -1 + \frac{1}{1+s} -5s -s^2/10\}, \quad s \ge 0.
\end{equation}
Using induction we can prove that for the $k$th derivative of the function $h(s)$
\begin{equation}\label{eqc5}
h^{(k)}(s)= h(s)(-1)^k \frac{1}{(1+s)^{2k}}P_{3k}(s),\quad k=1 ,2, \ldots,
\end{equation}
where $P_{3k}(s)$ is a polynomial in $s$ of degree no higher than $3k$ with non-negative coefficients. Thus, the function $h(s)$ is completely monotonic and, according to the theorem of S.N. Bernstein, is the Laplace transform of some finite measure on a non-negative semi-axis. Since $h(0)=1$, this is a probability measure. Hence,
\[h(s) =\int_{0}^{\infty}\exp\{-s x\}d\mathcal{A}(x). \]
Let us now define $f(t)=h(t^2)$,
\[ f(t) = \int_{0}^{\infty}\exp\{-t^2 x\}d\mathcal{A}(x). \]
Since $\exp\{-t^2 x\}$ for $x>0$ is a characteristic function of a normal law with zero mean, then $f(t)$ is a mixture of characteristic functions and, therefore, is itself a characteristic function of some probabilistic distributions. However, the power series expansion of its logarithm has only a finite number of different coefficients. Moreover, its logarithm
\[g(t)=
\frac{-6t^2-5.1t^4-0.1t^6}{1+t^2}\]
has a polynomial of the sixth degree in the numerator, and $1+t^2$ in the denominator.
\end{proof}

\end{document}